\documentclass[preprint,12pt]{elsarticle}

\usepackage{amssymb}
\usepackage{amsmath}
\usepackage{amsthm}
\usepackage{xcolor}
\newtheorem{theorem}{Theorem}[section]

\newtheorem{corollary}[theorem]{Corollary}

\newtheorem{remark}[theorem]{Remark}
\newtheorem{definition}[theorem]{Definition}

\journal{Nonlinear Analysis: Theory, Methods and Applications}

\begin{document}
	
	\begin{frontmatter}
		
		
		
		\title{Lower order term for the fractional Laplacian with a Hardy potential}
		
		\author[label 1]{Rub\'en Fi\~nana}\corref{correspondingauthor}
		\ead{rfa803@ual.es}
		\cortext[correspondingauthor]{Corresponding author.}
		
		\author[label 2]{Alexis Molino}
		\ead{amolino@ugr.es}
		
		

		\begin{abstract}
			This paper is concerned with the following semilinear elliptic equation involving the fractional Laplacian: $$(-\Delta)^s u+ g|u|^{p-1}u= \lambda \frac{u}{|x|^{2s}}+f(x),$$ in a bounded domain $\Omega$ of $\mathbb{R}^N\,(N>2s)$, subject to the zero Dirichlet condition in $\mathbb{R}^N\setminus \Omega$, where $0<g\in L_{loc}^1(\Omega)$, $p>1$ and $f\in L^{(p+1)/p}_g(\Omega)$. Under certain integrability  condition on $g$, the existence of solution  is proven for every $\lambda\in \mathbb{R}$. Moreover, the regularity of solution is also obtained.
		\end{abstract}
		
		
		
		
	\end{frontmatter}
	
		
		
		\section{Introduction}
		
		In recent years, semilinear elliptic equations involving the fractional \\Laplacian with the Dirichlet condition:
		\begin{equation}\label{problemaa}
			\begin{cases}
				\displaystyle(-\Delta)^s u = f(x,u), &  \text{in $\Omega$,}\\
				\mathnormal{u}=0, & \text{in $\mathbb{R}^N\setminus\Omega$,}
			\end{cases}
		\end{equation}
		with $s\in (0,1)$ and $N>2s$, have been extensively studied, constituting in itself a field of interest for the researcher in Nonlinear Analysis and PDE. There is a large amount of literature depending on the different cases of $f(x,u)$. For instance, when
		$f(x,u)=f(x)$  is a measurable function and $\Omega$ is a bounded Lipschitz domain, in Leonori et al. \cite{Leonori} the authors obtained results concerning the existence, uniqueness and summability of the solution with respect to the summability of $f$.
		
		There are also Poho\v{z}aev type results for the fractional Laplacian. In Fall and Weth \cite{FallWeth} the authors obtain nonexistence of positive solutions for a star-shaped domain with respect to the origin, $\Omega$ and $f: \overline{\Omega}\setminus \{0\} \times [0,+\infty)\rightarrow \mathbb{R}$ locally Lipschitz in the second variable and  supercritical in the sense that  the function $\lambda \rightarrow \lambda^{1-2^*_s}f(\lambda^{-2/(N-2s)}x,\lambda u)$ is nondecreasing on $[1,+\infty)$, a typical example is  $f(x,u)=|u|^{p-1}u$ for $p\geq 2^*_s-1$, where ${2^*_s=2N/(N-2s)}$. After that, there were results concerning changing-sign solutions with a more relaxed condition on the locally Lipschitz nonlinearities $f$. It is worth highlighting the work of Ros-Oton and Serra \cite{Ros} given a Poho\v{z}aev type identity for the fractional Laplacian. However,  if $\Omega$ is an annular type domain there is a nontrivial solution for $f(x,u)=|u|^{2^*_s-2}u$ (see Secchi et al. \cite{Secchi}). There is also a Brezis-Nirenberg type result for \eqref{problemaa}. To be more precise, for
		$f(x,u)=\lambda u^q+u^{2^*_s-1}$ with $q\in (1,2^*_s-1)$, it is guaranteed by Barrios et al. \cite{Barrios2} the existence of positive solutions for every $\lambda >0$.
		
		An important case for problem \eqref{problemaa} is the effect of the Hardy potential term, i.e., $f(x,u)=\lambda u/|x|^{2s}+f(x)$ with $0 \in \Omega$:
		\begin{equation}\label{problema4}
			\begin{cases}
				\displaystyle(-\Delta)^s u = \lambda \frac{u}{|x|^{2s}}+f(x), &  \text{in $\Omega$,}\\
				\mathnormal{u}=0, & \text{in $\mathbb{R}^N\setminus\Omega$.}
			\end{cases}
		\end{equation}  
		This equation arises in problems of relativistic matter stability in magnetic fields and was developed by the pioneering work of Abdellaoui et al. \cite{Ireneo}. As it occurred in Boccardo, Orsina and Peral \cite{Boc} for the local case ($s=1$), they obtained existence and regularity results for $0<\lambda \leq \Lambda_{N,s}$, where $\Lambda_{N,s}$ is a constant defined as:
		\[\Lambda_{N,s}= 2^{2s}\frac{\Gamma^2\left(\frac{N+2s}{4}\right)}{\Gamma^2\left(\frac{N-2s}{4}\right)}.\]
		Here $\Gamma$ is the usual gamma function. Besides 
		$$\lim_{s\rightarrow 1} \Lambda_{N,s}=\mathcal{H}^2=\left(\frac{N-2}{2}\right)^2,$$
		being $\mathcal{H}^2$ the classical Hardy constant.
		
		Furthermore, problem  \eqref{problema4}  is related to the relativistic Hardy inequality which was proved by Herbst in \cite{Her} (see also \cite{Barrios} and the references therein):
		
		For $u\in H^s_0(\Omega)$ and $0\in \Omega$,  it follows that
		\begin{equation}\label{DesigualdadHardy}
			\frac{c_{N,s}}{2}\int_{\mathbb{R}^N}\int_{\mathbb{R}^N} \frac{|u(x)-u(y)|^2}{|x-y|^{N+2s}}\geq \Lambda_{N,s}\int_{\Omega} \frac{u^2}{|x|^{2s}}
		\end{equation}		
		$$c_{N,s}:=\left(\displaystyle \int_{\mathbb{R}^N}\frac{1-\cos{(y_1)}}{|y|^{N+2s}}\right)^{-1}>0.$$
		The main purpose of this work is to improve and generalize the results of \eqref{problema4} obtained in \cite{Ireneo},  by adding a lower order term. Concretely, to study the existence and regularity of solutions for the problem:
		\begin{equation}\label{problemaoriginal}
			\begin{cases}
				\displaystyle(-\Delta)^s u+ g(x)|u|^{p-1}u= \lambda \frac{u}{|x|^{2s}}+f(x), &  \text{in $\Omega$,}\\
				\mathnormal{u}=0, & \text{in $\mathbb{R}^N\setminus\Omega$,}
			\end{cases}
		\end{equation}
		where $\Omega\subset \mathbb{R}^N$ ($N>2s$) is a bounded domain with smooth boundary and $0\in \Omega$, $s\in (0,1)$, $p>1$, $\lambda\in \mathbb{R}$ and $f\in L^{(p+1)/p}_g(\Omega)$, i.e. $|f|^{(p+1)/p}g\in L^1(\Omega)$. The function $g$ is a positive function in $L_{loc}^1(\Omega)$. The particular case $g\equiv 0$ is studied in \cite{Ireneo}. On the other hand, the case $g\equiv 1$ was developed by Mi et al. in \cite{Mi} for positive solutions and $f$ a nonnegative function. As can be seen below, the results for existence and summability of solutions obtained in this work improve those obtained in \cite{Mi}. Indeed, in Theorem \ref{teo1}, under the following mild integrability restriction:
		\begin{equation}\label{condicion}
			\frac{1}{|x|^{2s}g^{\frac{2}{p+1}}}\in L^{\frac{p+1}{p-1}}(\Omega).
		\end{equation}
		the existence of a solution to problem \eqref{problemaoriginal} is proven for every $\lambda \in \mathbb{R}$. Observe that condition \eqref{condicion}, with $g\equiv 1$, implies that $p>2^*_{s}-1$. Therefore, Theorem \ref{teo1} improves  \cite[Theorem 1.4]{Mi} where the authors find solutions for ${0<\lambda\leq\Lambda_{N,s}}$.
		
		Additionally, if  $f\in L^q_g(\Omega)$ for $q\geq (p+1)/p$ and the following further condition is satisfied
		\begin{equation}\label{cond3}
			\int_{\Omega} |x|^{\frac{2sp q}{1-p}}g^{1-\frac{p q}{p-1}}<+\infty,
		\end{equation}
		Theorem \ref{teoreg} guarantees that the solutions of \eqref{problemaoriginal} belong to $H^s_0(\Omega)\cap L_g^{pq}(\Omega)$ for every $\lambda\in\mathbb{R}$. In the case $g\equiv 1$, Theorem \ref{teo1} and Theorem \ref{teoreg} establish the existence of a solution in $H^s_0(\Omega)\cap L^{pq}(\Omega)$, for $q\geq (p+1)/p$ and $\lambda\in \mathbb{R}$ when $p>2^*_s-1$ (condition \eqref{condicion}) and $q<\frac{N}{2s}(1-\frac{1}{p})$ (condition \eqref{cond3}). As a consequence, there exists an improvement about the regularity of the solutions with respect to \cite[Theorem 1.5]{Mi}, where the solutions are in $H^s_0(\Omega)\cap L^{\frac{N(p(q-1)+1)}{N-2s}}(\Omega)$ for $0<\lambda<\Lambda_{N,s} 4p(q-1)/(p(q-1)+1)^2$. Moreover, their interval of $\lambda$, unlike ours, converges to zero as $p\rightarrow +\infty$.

		It is worth noting that the low regularity of $g$, as well as the absence of constraints of $\lambda$, allow the extension of the results obtained to a significant number of problems, see Corollaries \ref{twohardy} and \ref{problemboundedg}.
		
		This paper is organized as follows: Section 2 introduces the functional setting of problem \eqref{problemaoriginal} and provide some previous concepts. Section 3 covers the existence of solutions in  Theorem \ref{teo1} and Theorem \ref{propo} by imposing some conditions of integrability of $g$. Finally, in Section 4, Theorem \ref{teoreg} is devoted to the regularity of the solution when the summability of the function $f$ is increased under some additional assumptions. In this last section, the regularity effect of the lower order term is evidenced by relating the results to some previous works depending on the value of $\lambda$. Technical proofs for a fundamental inequality is provided in \ref{appendix:A}.
		
		\section{Preliminaries}
		In this section, we give basic notations, definitions, and tools that will be used
		throughout the paper. The fractional Laplacian operator, denoted as $(-\Delta)^s$, is defined by:
		\begin{equation*}
			(-\Delta)^su(x)=c_{N,s}\, \text{P.V.}\int_{\mathbb{R}^N} \frac{u(x)-u(y)}{|x-y|^{N+2s}} \quad x\in \mathbb{R}^N,
		\end{equation*}
		where $s\in(0,1)$, P.V. refers to the Cauchy principal value and the norma\-lization constant $c_{N,s}$ is given in \eqref{DesigualdadHardy}, consult \cite{DGV} for further details.

		One essential space associated to this nonlocal operator is the fractional Sobolev space 
		\[H^s(\mathbb{R}^N)=\left\{ u\in L^2(\mathbb{R}^N): \frac{u(x)-u(y)}{|x-y|^{N/2+s}}\in L^2(\mathbb{R}^N\times \mathbb{R}^N) \right\}.\]
		In particular, for the problem \eqref{problemaoriginal} the space $H_0^s(\Omega)$, where $\Omega \subset \mathbb{R}^N$, is considered. It is described as:
		\[H_0^s(\Omega)=\{u\in H^s(\mathbb{R}^N): u\equiv 0 \textrm{ a.e. in } \mathbb{R}^N\setminus \Omega \}.\]
		Moreover, $H_0^s(\Omega)$ is a Hilbert space with the scalar product
		\[
		\langle u, w \rangle = \frac {c_{N,s}}2\int_{\mathbb{R}^N}  \int_{\mathbb{R}^N} \frac{(u(x)-u(y))(w(x)-w(y))}{|x-y|^{N+2s}},\quad  u,w\in H_0^s(\Omega),
		\]
		and the norm 
		\[ \|u\|^2_{H_0^s(\Omega)}=\|(-\Delta)^\frac{s}{2}u\|^2_{L^2(\mathbb{R}^N)}.\] 
		In regards to the properties of this functional space, see \cite{DGV}, it is well known that in the case $u, w \in H_0^s(\Omega)$, then  
		\[\displaystyle \int_\Omega w (-\Delta)^s u= \int_\Omega((-\Delta)^s w ) u  = \int_{\mathbb{R}^N} (-\Delta)^\frac{s}{2}u (-\Delta)^\frac{s}{2}w=\langle u, w \rangle.\]
		For $r\geq 1$ and $0<g\in L_{loc}^1(\Omega)$, we denote by $L^{r}_g(\Omega)$ the linear space of all measurable functions $f$ such that $|f|^{r}g\in L^1(\Omega)$, endowed with the seminorm 
		\[ |f|_{L^{r}_g(\Omega)}=\left(\int_{\Omega}|f|^{r}g\right)^{\frac{1}{r}}.\]	
		The space $E$ is defined by:
		\[E= H^s_0(\Omega)\cap L^{p+1}_g(\Omega),\]
		with the associated norm
		\[ \|u\|_E=\|u\|_{H^s_0(\Omega)}+|u|_{L^{p+1}_g(\Omega)}.\]
		Additionally, for every $f\in L^{(p+1)/p}_g(\Omega)$, there exists a functional $v_f\in E^*$ such that  
		\begin{equation}\label{DualE}
			\langle v_f, h
			\rangle = \int_{\Omega} fhg,\textit{ $\forall h\in L^{p+1}_g(\Omega)$.}
		\end{equation}
		Indeed, by the H\"older inequality and the definition of $E$,
		\begin{align*}
			\langle v_f, h
			\rangle= \int_{\Omega} fhg&\leq \|fg^{\frac{p}{p+1}}\|_{L^{\frac{p+1}{p}}(\Omega)} \|hg^{\frac{1}{p+1}}\|_{L^{p+1}(\Omega)}\\&= \|f\|_{L_g^{\frac{p+1}{p}}(\Omega)} \|h\|_{L_g^{p+1}(\Omega)}\leq \|f\|_{L_g^{\frac{p+1}{p}}(\Omega)} \|h\|_{E},
		\end{align*}
		so $v_f\in E^*$.

		\begin{definition}\label{def:sol}
			A function $u \in E$ is a solution to the problem \eqref{problemaoriginal}, if for $0<g\in L_{loc}^1(\Omega)$ and $f\in L_g^{(p+1)/p}(\Omega)$,
			\[\int_{\mathbb{R}^N}(-\Delta)^{\frac{s}{2}}u(-\Delta)^{\frac{s}{2}}w + \int_{\Omega}|u|^{p-1}uwg-\lambda \int_{\Omega}\frac{u}{|x|^{2s}}w-\int_{\Omega}fwg=0,\]
			for all $w\in E$.
		\end{definition}
		
		In order to apply variational methods, note that these solutions can be characterized as critical points of the functional $\phi_{\lambda}\in C^1(E,\mathbb{R})$ defined by:
		\begin{equation*}
			\phi_{\lambda} (u)= \frac{1}{2}\|u\|^2_{H^s_0(\Omega)}+\frac{1}{p+1}\int_{\Omega}|u|^{p+1}g-\frac{\lambda}{2} \int_{\Omega}\frac{u^2}{|x|^{2s}}-\int_{\Omega}fug.
		\end{equation*}
		
		\section{Existence of solution}
		
		The main purpose of this section is to establish the existence of a solution to \eqref{problemaoriginal} for every $\lambda\in \mathbb{R}$. 
		The Hardy inequality \eqref{DesigualdadHardy} and the previously mentioned integrability condition \eqref{condicion} are essential elements to reach the result.
		
		Furthermore, from \eqref{condicion} it is deduced that if $u\in H^s_0(\Omega)\cap L^{p+1}_g(\Omega)$, then
		\begin{equation}\label{acotacion}
			\int_\Omega \frac{u^2}{|x|^{2s}} \leq C_1\left(\int_{\Omega}|u|^{p+1}g\right)^{\frac{2}{p+1}}<+\infty,
		\end{equation}
		i.e., $u^2g^{2/(p+1)}\in L^{(p+1)/2}(\Omega)$.
		
		Indeed, by using the H\"older inequality with exponent $(p+1)/2$ and \eqref{condicion}, 
		\begin{align*}
			\int_\Omega \frac{u^2}{|x|^{2s}}=&\int_{\Omega} \frac{u^2g^{\frac{2}{p+1}}}{|x|^{2s}g^{\frac{2}{p+1}}}\leq 
			\| u^2g^{\frac{2}{p+1}}\|_{L^{\frac{p+1}{2}}(\Omega)} \left\| \frac{1}{|x|^{2s}g^{\frac{2}{p+1}}} 
			\right\|_{L^{\frac{p+1}{p-1}}(\Omega)}
			\\&\quad=C_1\left(\int_{\Omega}|u|^{p+1}g\right)^{\frac{2}{p+1}}< +\infty.
		\end{align*}
		Therefore 
		\begin{equation*}
			\int_\Omega \frac{u^2}{|x|^{2s}} \leq  C_1\left(\int_{\Omega}|u|^{p+1}g\right)^{\frac{2}{p+1}}< +\infty.
		\end{equation*}
		Along with condition \eqref{acotacion} and Hardy inequality \eqref{DesigualdadHardy}, a consequence of the Variational Principle of Ekeland  \cite[Theorem 7.2]{Cha} is needed.
		\begin{theorem}\label{ekeland}
			If $X$ is a Banach space and $\phi\in C^1(X,\mathbb{R})$ is bounded from below, then there exists a minimizing sequence $u_n$ for $\phi$ such that $\phi(u_n)\rightarrow \inf_X(\phi)$ and $\phi'(u_n)\rightarrow 0$ in $X^*$ as $n\rightarrow+\infty$.
		\end{theorem}
		This theorem establishes the existence of a critical point of $\phi_{\lambda}$, i.e. a solution to problem \eqref{problemaoriginal}, in the absence of the lower semicontinuity of $\phi_{\lambda}$.
		\begin{theorem}\label{teo1}
			Let $\Omega\subset \mathbb{R}^N$ be a smooth bounded domain such that $0\in \Omega$,  $f \in L^{(p+1)/p}_g(\Omega)$ with $g$ satisfying  inequality \eqref{condicion}. Then, for every $\lambda\in \mathbb{R}$ there exists a solution $u\in E$ to the problem \eqref{problemaoriginal}. 
			
		\end{theorem}
		\begin{proof}
			With the aim of guaranteeing the existence of a solution,  the coercivity and the lower boundedness of the functional $\phi_{\lambda}$ are needed.\par
			There are two possible scenarios depending on the value of $\lambda$:
			\begin{enumerate}
				\item[a)] Case $\lambda>0$.\par 
				Because of \eqref{acotacion} and \eqref{DualE} for $u\in E$, 
				\begin{align*}
					\phi_{\lambda} (u)\geq&  \frac{1}{2}\|u\|^2_{H^s_0(\Omega)}+\frac{1}{p+1}\int_{\Omega}|u|^{p+1}g-\frac{\lambda}{2}C_1\left(\int_{\Omega}|u|^{p+1}g\right)^{\frac{2}{p+1}}\\&-\langle f,u\rangle\\
					\geq& \frac{1}{2}\|u\|^2_{H^s_0(\Omega)}+\frac{1}{p+1}\int_{\Omega}|u|^{p+1}g-\frac{\lambda}{2} C_1\left(\int_{\Omega}|u|^{p+1}g\right)^{\frac{2}{p+1}}\\&-\| f\|_{E^*}\|u\|_{E}\\
					=&\frac{1}{2}\|u\|^2_{H^s_0(\Omega)}+\frac{1}{p+1}|u|^{p+1}_{L^{p+1}_g(\Omega)}-\frac{\lambda}{2}C_1|u|^2_{L^{p+1}_g(\Omega)}-\| f\|_{E^*}\|u\|_{E}.
				\end{align*}
				
				\item[b)] Case $\lambda\leq 0$.\par 
				It is enough to use \eqref{DualE} for $u\in E$, 
				\begin{align*}
					\phi_{\lambda} (u)&= \frac{1}{2}\|u\|^2_{H^s_0(\Omega)}+\frac{1}{p+1}\int_{\Omega}|u|^{p+1}g-\frac{\lambda}{2} \int_{\Omega}\frac{u^2}{|x|^{2s}}-\langle f,u\rangle\\
					&\geq \frac{1}{2}\|u\|^2_{H^s_0(\Omega)}+\frac{1}{p+1}|u|^{p+1}_{L^{p+1}_g(\Omega)}-\| f\|_{E^*}\|u\|_{E}.
				\end{align*}
			\end{enumerate}
			So for $\lambda\in \mathbb{R}$, $\phi_{\lambda}$ is coercive, yielding its lower boundedness. The claim is thus proved.

			Consequently, Theorem \ref{ekeland} can be applied to obtain a minimizing sequence, denoted by $u_n\in E$, such that 
			\[
			\phi_{\lambda}(u_n)\rightarrow \inf_{E}{\phi_{\lambda}},
			\]
			and for all $w\in E$, \begin{equation}\label{convergencia}
				\langle \phi'_{\lambda}(u_n),w\rangle\rightarrow 0,
			\end{equation}
			where 
			\begin{align*}
				\langle \phi'_{\lambda}(u_n),w\rangle=&\int_{\mathbb{R}^N}(-\Delta)^{\frac{s}{2}}u_n(-\Delta)^{\frac{s}{2}}w + \int_{\Omega}|u_n|^{p-1}u_nwg
				\\&-\lambda \int_{\Omega}\frac{u_n}{|x|^{2s}}w-\int_{\Omega}fwg.
			\end{align*}
			Due to the coercivity of the functional, $u_n$ is bounded in $E$, thus there exists a subsequence, still denoted by $u_n$, such that $u_n\rightharpoonup u$ in $E$. Therefore,
			\begingroup 
			\renewcommand{\theenumi}{(\alph{enumi})}
			\renewcommand{\labelenumi}{\theenumi}
			\begin{enumerate}
				\item $u_n\rightharpoonup u$ in $H^s_0(\Omega)$\label{conv1}.
				\item $u_n\rightharpoonup u$ in $L^{p+1}_g(\Omega)$\label{conv2}.
			\end{enumerate}
			\endgroup
			Taking into account \eqref{convergencia}, the limit $u$ is a solution of \eqref{problemaoriginal} if the following claim is satisfied
			\begin{equation}\label{limitsolution}
				\langle \phi'_{\lambda}(u_n),w\rangle\rightarrow \langle \phi'_{\lambda}(u),w\rangle= 0.
			\end{equation}
			Indeed, the convergence \ref{conv1} directly implies that 
			\begin{equation}\label{limitsolution1}
				\int_{\mathbb{R}^N}u_n(-\Delta)^{s}w\rightarrow \int_{\mathbb{R}^N}u(-\Delta)^{s}w.
			\end{equation}
			Moreover, by \ref{conv1}, $u_n\rightarrow u$ in $L^m(\Omega)$ for $m\in [1,2^*_s)$, so $u_n \rightarrow u$ a.e., therefore  $|u_n|^{p-1}u_n\rightarrow |u|^{p-1}u$ a.e.
			On the other hand, due to \ref{conv2} $u_n$ is bounded in $L^{p+1}_g(\Omega)$, thus $|u_n|^{p-1}u_n g^{p/(p+1)}$ is bounded in $ L^{(p+1)/p}(\Omega)$. That fact, along with convergence a.e., implies that $|u_n|^{p-1}u_n g^{p/(p+1)}\rightharpoonup |u|^{p-1}u g^{p/(p+1)}$ in $L^{(p+1)/p}(\Omega)$. As a consequence
			\begin{equation}\label{limitsolution2}
				\int_{\Omega}|u_n|^{p-1}u_nwg\rightarrow  \int_{\Omega}|u|^{p-1}uwg.
			\end{equation}
			Finally, consider $G_w: H^s_0(\Omega)\rightarrow \mathbb{R}$ defined as
			\[G_w(u)=\int_{\Omega}\frac{u}{|x|^{2s}}w.\]
			Notice that, if the H\"older inequality and Hardy inequality \eqref{DesigualdadHardy} are applied,
			\[\int_{\Omega}\frac{u}{|x|^{2s}}w\leq \left\|\frac{u}{|x|^{s}}\right\|_{L^2(\Omega)} \left\|\frac{w}{|x|^{s}}\right\|_{L^2(\Omega)} \]\[\leq \frac{1}{\Lambda_{N,s}}\|u\|_{H^s_0(\Omega)}\|w\|_{H^s_0(\Omega)}.\]
			Hence, $G_w$ is a continuous linear functional and, by \ref{conv1}, ${G_w(u_n)\rightarrow G_w(u)}$. This implies that
			\begin{equation}\label{limitsolution3}
				\int_\Omega \frac{u_n}{|x|^{2s}}w\rightarrow \int_\Omega \frac{u}{|x|^{2s}}w.
			\end{equation}
			As a result of the limits \eqref{limitsolution1},  \eqref{limitsolution2} and  \eqref{limitsolution3}, the claim \eqref{limitsolution} is proved.
		\end{proof}

		Although $u$ is a critical point of $\phi_{\lambda}$ and therefore a solution of $\eqref{problemaoriginal}$, it does not necessarily correspond to a minimum of the functional $\phi_\lambda$. In order to obtain this correspondence,  an additional assumption is necessary:
		\begin{equation}\label{Cond2}
			\int_{\Omega}|x|^{\frac{2ts}{2-t}}g^{\frac{2t}{(p+1)(2-t)}}<+\infty \text{ for some $t\in (2,p+1)$},
		\end{equation}
		which guarantees the lower semicontinuity of $\phi_\lambda.$
		
		\begin{remark}
			Note that \eqref{Cond2} implies \eqref{condicion}. Indeed, condition \eqref{Cond2} can be rewritten as: 
			\[\frac{1}{|x|^{2s}g^{\frac{2}{p+1}}}\in L^{\frac{t}{t-2}}(\Omega) \text{ with $t\in (2,p+1)$}.\]
			On the other hand, $t/(t-2)>(p+1)/(p-1)$, so $L^{t/(t-2)}(\Omega)\subset L^{(p+1)/(p-1)}(\Omega)$. As a consequence, \eqref{condicion}  is satisfied.
		\end{remark}
		
		\begin{theorem}\label{propo}
			
			If condition \eqref{Cond2} is satisfied, then the solution of \eqref{problemaoriginal} obtained in Theorem \ref{teo1}, is a minimum of the functional $\phi_{\lambda}$.
		\end{theorem}

		\begin{proof}
			As it is commented before, the lower semicontinuity of the functional $\phi_{\lambda}$ is required, so if $u_n\rightharpoonup u \in E$, then 			
			\begin{align}\label{solutionminimum}
				\liminf_{n\rightarrow +\infty} &\left\{\frac{1}{2}\|u_n\|^2_{H^s_0(\Omega)}+\frac{1}{p+1}|u_n|^{p+1}_{L^{p+1}_g(\Omega)}-\frac{\lambda}{2}\int_{\Omega}\frac{u_n^2}{|x|^{2s}}-\int_{\Omega}fu_ng \right\}\\
				& \geq \frac{1}{2}\|u\|^2_{H^s_0(\Omega)}+\frac{1}{p+1}|u|^{p+1}_{L^{p+1}_g(\Omega)}-\frac{\lambda}{2}\int_{\Omega}\frac{u^2}{|x|^{2s}}-\int_{\Omega}fug.
				\nonumber
			\end{align}
			Due to the weak convergence of $u_n$ in $E$ and \eqref{DualE}, 
			\begin{equation}\label{solutionminimum1}
				\int_{\Omega}fu_ng \rightarrow \int_{\Omega}fug.
			\end{equation}
			Furthermore, the weak lower semicontinuity of the norm $\|\cdot\|_{H^s_0(\Omega)}$ implies that
			\begin{equation}\label{solutionminimum2}
				\liminf_{n\rightarrow +\infty}\left\{\frac{1}{2}\|u_n\|^2_{H^s_0(\Omega)}\right\}\geq \frac{1}{2}\|u\|^2_{H^s_0(\Omega)}.
			\end{equation}
			Additionally, if $u_n\rightharpoonup u \in E$,  $|u_n|^{p+1}g$ is bounded in $L^{1}(\Omega)$, which are non-negative functions. So, the hypothesis of the Fatou Lemma holds, satisfying that
			\begin{equation}\label{solutionminimum3}
				\liminf_{n\rightarrow +\infty}\left\{\frac{1}{p+1}|u_n|^{p+1}_{L^{p+1}_g(\Omega)}\right\} \geq \frac{1}{p+1}|u|^{p+1}_{L^{p+1}_g(\Omega)}.
			\end{equation}	
			As a final step, the Lebesgue Dominated Convergence Theorem guarantees that 
			\begin{equation}\label{solutionminimum4}
				\frac{\lambda}{2}\int_{\Omega}\frac{u_n^2}{|x|^{2s}}\rightarrow \frac{\lambda}{2}\int_{\Omega}\frac{u^2}{|x|^{2s}}.
			\end{equation}
			Indeed, by the Sobolev embedding, if  $u_n\rightharpoonup u$ in $H^s_0(\Omega)$, then $u_n\rightarrow u$ in $L^m(\Omega)$ with $m\in [1,2_s^*)$, so $u_n\rightarrow u$ a.e., and $|u_n|^{t}g^{t/(p+1)}\rightarrow |u|^{t}g^{t/(p+1)}$ a.e.\par 
			Because of the weak convergence of $u_n\in E$,  $u_n\rightharpoonup u $ in $L^{p+1}_g(\Omega)$, so $u_ng^{1/(p+1)}$ is bounded in $L^{p+1}(\Omega)$. Moreover, for $t\in [1,p+1)$ the H\"older inequality implies the uniform integrability of $|u_n|^{t}g^{\frac{t}{p+1}}$:      
			\begin{align*}
				\lim_{|E|\rightarrow 0}\int_{E}|u_n|^{t}g^{\frac{t}{p+1}}&=	\lim_{|E|\rightarrow 0}\int_{\Omega}|u_n|^{t}g^{\frac{t}{p+1}} \chi_{E}				
				\\&\leq \lim_{|E|\rightarrow 0}\||u_n|^{t}g^{\frac{t}{p+1}}\|_{L^{\frac{p+1}{t}}(\Omega)}\|\chi_E\|_{L^{\frac{p+1}{p+1-t}}(\Omega)}\\
				&=	\lim_{|E|\rightarrow 0}\||u_n|g^{\frac{1}{p+1}}\|^t_{L^{p+1}(\Omega)}|E|^{1-\frac{t}{p+1}}\\&\leq C\lim_{|E|\rightarrow 0}|E|^{1-\frac{t}{p+1}}=0 .
			\end{align*}  	
			Therefore, due to the Vitali Theorem,
			\begin{equation}\label{solutionminimum4.1}
				\int_{\Omega}|u_n|^{t}g^{\frac{t}{p+1}}\rightarrow \int_{\Omega}|u|^{t}g^{\frac{t}{p+1}}.
			\end{equation}
			Furthermore, as a consequence of \eqref{solutionminimum4.1}, $u_ng^{1/(p+1)}$ has strong convergence in $L^{t}(\Omega)$ for $t\in [1,p+1)$. This guarantees the existence of a function $H\in L^{t}(\Omega)$ such that $|u_ng^{1/(p+1)}|\leq H$.
			
			Thus, $u_n^2/|x|^{2s}$ is bounded by $H^2/(|x|^{2s}g^{2/(p+1)})$.
			Besides, through the use of the H\"older inequality with exponent $t/2$ for $t\in (2,p+1)$, 
			\[\int_{\Omega}\frac{H^2}{|x|^{2s}g^{\frac{2}{p+1}}}\leq 
			\left(\int_{\Omega}\frac{1}{|x|^{\frac{2st}{t-2}}g^{\frac{2t}{(t-2)(p+1)}}}\right)
			^{\frac{t-2}{t}}
			\left(\int_{\Omega}H^{t}\right)^{\frac{2}
				{t}}.\]
			This expression is bounded by  condition \eqref{Cond2} and by the fact that $H\in L^{t}(\Omega)$.
			
			Moreover, $u_n^2/|x|^{2s}\rightarrow u^2/|x|^{2s}$ a.e., by the convergence a.e. of $u_n$, so the Lebesgue Dominated Convergence Theorem  provides  the limit \eqref{solutionminimum4}.

			Hence, from \eqref{solutionminimum1}, \eqref{solutionminimum2}, \eqref{solutionminimum3} and \eqref{solutionminimum4}, \eqref{solutionminimum} follows. 
			
		\end{proof}

		\section{Regularity of the solution}
		This section analyses the regularity achieved by the solution $u$ to the problem \eqref{problemaoriginal} provided by Theorem \ref{teo1}, depending on the summability of the function $f$.\par 
		In Theorem \ref{teo1} and Theorem \ref{propo} the solution $u$ belongs to $E$ when ${f\in L^{(p+1)/p}_g(\Omega)}$. The next result establishes that the summability of $u$ can be improved as long as the regularity of $f$ increases due to the condition \eqref{cond3}.
		\begin{theorem} \label{teoreg}
			Assume that $g\in L^1_{loc}(\Omega)$, $g>0$ a.e. in $\Omega$, $\lambda$ is a real constant, \eqref{cond3} is satisfied and $q\geq (p+1)/p$ such that $f\in L^q_g(\Omega)\cap L^{(p+1)/p}_g(\Omega)$.	If $u$ is a solution to \eqref{problemaoriginal}, then $u\in L^{p q}_g(\Omega)$.
		\end{theorem}
		
		\begin{proof}
			Let $u\in E$ be a solution of \eqref{problemaoriginal}, then the theorem is proved if the following claim holds true:
			\begin{equation}\label{resultado}
				\int_\Omega |u|^{p q}g<+\infty.
			\end{equation}	
			With the aim of verifying this claim,  denote $\rho=p(q-1)-1$ which is positive since $q\geq (p +1)/p $, and define for every $k>0$ a function $T_k:\mathbb{R}\rightarrow \mathbb{R}$ as:
			\begin{equation*} T_k(x)=
				\left\{ \begin{array}{ll}
					k, & \text{if $x>k$,} \\
					
					x, & \text{if $|x|\leq k$,} \\
					-k, & \text{if $x<-k.$} 
				\end{array}
				\right.
			\end{equation*}	
			Owing to this definition, $|T_k(u)|^{\rho}T_k(u)$ can be regarded as a test function; therefore, the next equality is satisfied:
			\begin{align*}
				\int_{\mathbb{R}^N}&(-\Delta)^{\frac{s}{2}}u(-\Delta)^{\frac{s}{2}}(|T_k(u)|^{\rho}T_k(u))+\int_{\Omega}|u|^{p -1}u|T_k(u)|^{\rho}T_k(u)g\\&
				-\lambda\int_{\Omega}\frac{u}{|x|^{2s}}|T_k(u)|^{\rho}T_k(u)-\int_{\Omega}fg|T_k(u)|^{\rho}T_k(u)=0.
			\end{align*}
			Moreover,  $(-\Delta)^{s/2}u(-\Delta)^{s/2}(|T_k(u)|^{\rho}T_k(u))\geq 0$, see \ref{appendix:A}. Thus,
			\begin{align*}
				\int_{\Omega}|u|^{p -1}u|T_k(u)|^{\rho}T_k(u)g\leq& \lambda\int_{\Omega}\frac{u}{|x|^{2s}}|T_k(u)|^{\rho}T_k(u)
				+\int_{\Omega}fg|T_k(u)|^{\rho}T_k(u).
			\end{align*}
			Since $T_k(u)u=|T_k(u)u|$, the above inequality is equivalent to
			\begin{equation}\label{exfinal}
				\int_{\Omega}|u|^{p }|T_k(u)|^{\rho+1}g\leq\lambda\int_{\Omega}\frac{|u|}{|x|^{2s}}|T_k(u)|^{\rho+1}+\int_{\Omega}fg|T_k(u)|^{\rho}T_k(u).
			\end{equation}
			The next step of the proof is to bound each term of the inequality and obtain \eqref{resultado}. Let
			\[ 
			F_k(u)=|u|^{p -\theta}|T_k(u)|^{\rho+\theta+1}g,
			\]	
			where $\theta=(1+\rho)(p -1)/(\rho+2)=p(q-1)(p-1)/(p(q-1)+1)\in (0,p-1)$.
			
			There exist two options taking into account the sign of $\lambda$:
			
			\begin{enumerate}
				\item[a)] Case $\lambda>0$.\par 
				Due to the definition of $T_k$, $|T_k(x)|\leq |x|$ $\forall x \in \mathbb{R}$, hence
				\[|u|^{p }|T_k(u)|^{\rho+1}g=\frac{F_k(u)|T_k(u)|^{-\theta}}{|u|^{-\theta}}\geq F_k(u),\]
				then,
				\begin{equation}\label{ex1}
					\int_{\Omega} |u|^{p }|T_k(u)|^{\rho+1}g\geq \int_{\Omega}F_k(u).
				\end{equation}
				Moreover, owing to the fact that $1+\theta+\rho=(1+\rho)(p -\theta)$, the definition of $\theta$ and the H\"older inequality with exponent $p -\theta>1$, 		
				\begin{align*}
					\lambda\int_{\Omega}\frac{|u|}{|x|^{2s}}|T_k(u)|^{\rho+1}&=\lambda\int_{\Omega}\frac{|u|g^{\frac{1}{p -\theta}}}{|x|^{2s}g^{\frac{1}{p -\theta}}}|T_k(u)|^{\rho+1}\\&\leq \lambda \left\|\frac{1}{|x|^{2s}g^{\frac{1}{p -\theta}}}\right\|_{L^{\frac{p q}{p -1}}(\Omega)}\||u|g^{\frac{1}{p -\theta}}|T_k(u)|^{\rho+1}\|_{L^{p -\theta}(\Omega)}.
				\end{align*}
				Finally, as a consequence of $\eqref{cond3}$, there exists some $C_3>0$ such that
				\begin{align}\label{ex2}
					\lambda \left\|\frac{1}{|x|^{2s}g^{\frac{1}{p -\theta}}}\right\|_{L^{\frac{p q}{p -1}}(\Omega)}&\| |u|g^{\frac{1}{p -\theta}}|T_k(u)|^{\rho+1}\|_{L^{p -\theta}(\Omega)}
				\end{align}
				\[\leq C_3 \left(\int_{\Omega}F_k(u)\right)^{\frac{1}{p -\theta}}.\]
				In addition, making use of the H\"older inequality with exponent $q$,
				
				\begin{align*}
					\int_{\Omega}fg|T_k(u)|^{\rho}T_k(u)&=\int_{\Omega}fg^{\frac{1}{q}}|T_k(u)|^{\rho}T_k(u)g^{\frac{q-1}{q}}
					\\&\leq\|fg^{\frac{1}{q}}\|_{L^q(\Omega)}\||T_k(u)|^{\rho+1}g^{\frac{q-1}{q}}\|_{L^{\frac{q}{q-1}}(\Omega)}.
				\end{align*}
				Since $f\in L^q_g(\Omega)\cap L^{(p+1)/p}_g(\Omega)$ with $q\geq (p+1)/p$, it holds that for some $C_4>0$,
				\begin{align*}
					\|fg^{\frac{1}{q}}\|_{L^q(\Omega)}\||T_k(u)|^{\rho+1}g^{\frac{q-1}{q}}\|_{L^{\frac{q}{q-1}}(\Omega)}  \leq& C_4\||T_k(u)|^{\rho+1}g^{\frac{q-1}{q}}\|_{L^{\frac{q}{q-1}}(\Omega)}\\
					& = C_4 \left( \int_{\Omega} |T_k|^{\frac{q(\rho+1)}{q-1}}g\right)^{\frac{q-1}{q}}.
				\end{align*}
				Taking into account that $q(\rho+1)/(q-1)=\rho +1+p$, and $|T_k(x)|\leq |x|$,  it follows that
				\begin{align*}
					&C_4 \left( \int_{\Omega} |T_k|^{\frac{q(\rho+1)}{q-1}} g\right) = C_4 \left( \int_{\Omega} |T_k|^{\rho +1+\theta-\theta + p} g\right)
					\\ &\leq C_4 \left( \int_{\Omega} |T_k|^{\rho +1+\theta}|u|^{p-\theta} g\right)= C_4 \left( \int_{\Omega} F_k(u)\right). 
				\end{align*}
				Then, 
				\begin{equation}\label{ex3}
					\int_{\Omega}fg|T_k(u)|^{\rho}T_k(u)\leq C_4\left( \int_{\Omega} F_k(u)\right)^{\frac{q-1}{q}}.
				\end{equation}
				Therefore, by the inequalities \eqref{ex1}, \eqref{ex2} and \eqref{ex3}, \eqref{exfinal} can be rewritten as 
				\[\int_{\Omega}F_k(u)\leq  C_3 \left(\int_{\Omega}F_k(u)\right)^{\frac{1}{p -\theta}}+C_4\left(\int_{\Omega}F_k(u)\right)^{\frac{q-1}{q}}.\]		
				Since $1/(p -\theta)= (pq-p+1)/pq$ and $(q-1)/q$ are less than 1, there exist positive constants $c$, $C_5$ such that 
				\[\int_{\Omega}|u|^{p -\theta}|T_k(u)|^{\rho+\theta+1}=\int_{\Omega}F_k(u)\leq C_5, \text{ $\forall k\geq c$}.\]
				As a final step, notice that $\rho+1+p =p q$, so by making $k\rightarrow +\infty$ the Fatou lemma yields \eqref{resultado},
				\[\int_{\Omega}|u|^{p q}g=\int_{\Omega}|u|^{\rho+1+p }g\leq C_5.\]
				\item[b)] Case $\lambda\leq 0$.\par 
				This case can be argued analogously to the case $\lambda >0$, since from \eqref{exfinal} it is achieved that
				\[\int_{\Omega}|u|^{p -1}u|T_k(u)|^{\rho}T_k(u)g\leq \int_{\Omega}fg|T_k(u)|^{\rho}T_k(u).\] 
			\end{enumerate}
		\end{proof}
		\begin{remark}\label{summabilityintegralg}
			Regarding the summability of $f$, note that if ${g\in L^1(\Omega)}$, the condition $f\in L^q_g(\Omega)$ is sufficient, as it ensures that $f$ belongs to $L^{(p+1)/p}_g(\Omega)$. \par
			Indeed, consider $0<c\in \mathbb{R}$ such that $1/q+c=p/(p+1)$ and the following equality 
			\[\int_{\Omega}f g^{\frac{p}{p+1}}g^{-c}g^{c}=\int_{\Omega}f g^{\frac{1}{q}}g^{c}.\]	
			By applying H\"older inequality with exponent $q$ and the fact that $f\in L^{q}_g(\Omega)$,
			\[\int_{\Omega}f g^{\frac{1}{q}}g^{c}\leq \|f g^{\frac{1}{q}}\|_{L^q(\Omega)}\|g^c\|_{L^{\frac{q}{q-1}}(\Omega)} \leq |f|_{L^{q}_g(\Omega)} \|g^c\|_{L^{\frac{q}{q-1}}(\Omega)}.\]	
			In addition, since $g\in L^1(\Omega)$, then $g^c\in L^{1/c}(\Omega)\subset L^{q/(q-1)}(\Omega)$, because $1/c\geq q/(q-1)$. Thus, $\int_{\Omega}f g^{p/(p+1)}g^{-c}g^{c}<+\infty$ and $f\in L^{(p+1)/p}_g(\Omega)$.
			
		\end{remark}
		
		As an immediate consequence of Theorem \ref{teo1} and Theorem \ref{teoreg}, the following corollary arises.

		\begin{corollary}\label{finalregularity}
			Consider $q\geq (p+1)/p$ such that $f\in L^q_g(\Omega)\cap L^{(p+1)/p}_g(\Omega)$ and that conditions \eqref{condicion} and \eqref{cond3} are satisfied. 
			Then, problem \eqref{problemaoriginal} has a solution $u\in H^s_0(\Omega)\cap L^{p q}_g(\Omega) $ for every $\lambda \in \mathbb{R}$.
		\end{corollary}
		There are various notable problems, depending on the choice of function $g$, to which Theorem \ref{teo1} and Theorem \ref{teoreg} can be applied. For instance, the next corollary deals with problems involving two Hardy potentials. 
		\begin{corollary}\label{twohardy}
			If $\mu>0$ and $f\in L^{q}_{\mu/|x|^{(p+1)s}}(\Omega)$, with $q\geq (p+1)/p$ and $1<p<(N-s)/s$. Then, the problem 
			\begin{equation}\label{problemtwohardy}
				\begin{cases}
					\displaystyle(-\Delta)^s u+ \mu\frac{|u|^{p-1}u}{|x|^{(p+1)s}}= \lambda \frac{u}{|x|^{2s}}+f(x), &  \text{in $\Omega$,}\\
					u=0, & \text{in $\mathbb{R}^N\setminus\Omega$,}
				\end{cases}
			\end{equation}
			has a solution which belongs to $ H^s_0(\Omega)\cap L^{pq}_{\mu/|x|^{(p+1)s}}(\Omega),$ for every $\lambda\in \mathbb{R}$.
		\end{corollary}
		\begin{proof}
			Due to the assumptions imposed on $p$, it is guaranteed that the function $\mu/|x|^{(p+1)s}$, denoted as $g$, is a positive function in $L^1(\Omega)$. On the other hand, condition \eqref{condicion} is satisfied for this function $g$ since
			\[
			\int_{\Omega} C(\mu)|x|^{\frac{2(p+1)s}{1-p}}|x|^{\frac{-2(p+1)s}{1-p}}=C(\mu)|\Omega|
			\]
			where $C(\mu)$ refers to a constant depending on $\mu$. Furthermore, taking into account that $q\geq (p+1)/p$, condition \eqref{cond3} also holds
			\[
			\int_{\Omega} C(\mu)|x|^{\frac{2spq}{1-p}}|x|^{(-(p+1)s)(1-\frac{pq}{p-1})} =C(\mu)\int_{\Omega} |x|^{s(pq-(p+1))}<+\infty.
			\]
			As a consequence, by Remark \ref{summabilityintegralg} and Corollary \ref{finalregularity} the proof is completed.
		\end{proof}

		Finally, it is worth highlighting the influence of the lower order term concerning the regularity and existence of the solutions to other problems in the absence of it.
		
		\begin{corollary}\label{problemboundedg}
			Consider $g\in L^1(\Omega)$ such that $g\geq k>0$ a.e. in $\Omega$, $p>2^*_s-1$ and $(p+1)/p\leq q<\frac{N}{2s}(1-\frac{1}{p})$. Hence, there exists a solution ${u\in L_g^{pq}(\Omega)\cap H^s_0(\Omega)}$ for every value of $\lambda\in \mathbb{R}$ to
			\begin{equation*}
				\begin{cases}
					\displaystyle(-\Delta)^s u+ g|u|^{p-1}u=\lambda\frac{u}{|x|^{2s}}+ f(x), &  \text{in $\Omega$,}\\
					u=0, & \text{in $\mathbb{R}^N\setminus\Omega$,}
				\end{cases}
			\end{equation*}	
			with $f\in L_g^q(\Omega)$.\par
		\end{corollary}
		\begin{proof}
			This result is deduced from applying Theorem \ref{teo1} to obtain the existence of a solution for $\lambda\in \mathbb{R}$ and Theorem \ref{teoreg} to achieve its regularity in $L^{pq}_g(\Omega)\cap H^s_0(\Omega)$. Therefore, it is enough to prove that conditions \eqref{condicion} and \eqref{cond3} are satisfied. Indeed, since $q<\frac{N}{2s}(1-\frac{1}{p})$,
			\[\int_{\Omega}\frac{1}{|x|^{\frac{2spq}{p-1}}g^{\frac{pq-p+1}{p-1}}}\leq C(k)\int_{\Omega}\frac{1}{|x|^{\frac{2spq}{p-1}}}<+\infty,\]
			condition \eqref{cond3} holds true. On the other hand, for  $p>2^*_s-1$, it follows that
			\[\int_{\Omega}\frac{1}{|x|^{\frac{2s(p+1)}{p-1}}g^{\frac{2}{p-1}}}\leq C(k)\int_{\Omega}\frac{1}{|x|^{\frac{2s(p+1)}{p-1}}}<+\infty.\]
			So condition \eqref{condicion} is also satisfied.
		\end{proof}
		This corollary can be interpreted as follows: for $g\geq k>0$ a.e. in $\Omega$, taking $p=1/\varepsilon$ for $\varepsilon>0$ small enough, it is found that the solution belongs to $L_g^{q/\varepsilon}(\Omega)$ for every $f\in L_g^q(\Omega)$ with $q\in [1+\varepsilon, N/2s(1-\varepsilon))$. As a consequence, an improvement over the regularity and the existence of solutions is achieved with respect to \cite[Theorem 4.2, Theorem 4.9]{Ireneo} for $\lambda\neq 0$ and  \cite[Theorem 16, Theorem 24]{Leonori} for $\lambda=0$, where there is no lower order term.
		\appendix
		\section{}
		\label{appendix:A}
		In this appendix, the following inequality is proved
		\begin{equation}\label{eq:AppendixA}
			(-\Delta)^{\frac{s}{2}}u(-\Delta)^{\frac{s}{2}}(|T_k(u)|^{\rho}T_k(u))\geq 0.
		\end{equation}
		According to the definition of the fractional Laplacian operator:
		\begin{align*}
			&(-\Delta)^{\frac{s}{2}}u(x)(-\Delta)^{\frac{s}{2}}(|T_k(u(x))|^{\rho}T_k(u(x)))\\&=\frac{{c_{N,s}}}{2}\int_{\Omega}\frac{(u(x)-u(y))(|T_k(u(x))|^{\rho}T_k(u(x))-|T_k(u(y))|^{\rho}T_k(u(y)))}{|x-y|^{N+2s}},
		\end{align*}
		where
		\begin{displaymath} |T_k(x)|^{\rho}T_k(x)=
			\left\{ \begin{array}{ll}
				k^{\rho}k, & \mbox{ $x>k$,} \\
				
				|x|^{\rho}x, & \mbox{ $|x|\leq k$,} \\
				-k^{\rho}k, & \mbox{ $x<-k.$} 
			\end{array}
			\right.
		\end{displaymath}
		The expression \eqref{eq:AppendixA} is a consequence of the increasing character of $|T_k(\cdot)|^{\rho}T_k(\cdot)$. \par Note that \eqref{eq:AppendixA} is immediate when $u(x)$ and/or $u(y)$, are greater than $k$ or less than $-k$. The remaining cases can be deduced from the analysis of the two following cases.
		\begin{enumerate}
			\item If $|u(x)|\leq k$ and $|u(y)|\leq k$.\par
			The expression $\eqref{eq:AppendixA}$ follows if
			\begin{align}\label{eq2:AppendixA}
				&(u(x)-u(y))(|T_k(u(x))|^{\rho}T_k(u(x))-|T_k(u(y))|^{\rho}T_k(u(y)) )
				\\&=(u(x)-u(y))( |u(x)|^{\rho}u(x)-|u(y)|^{\rho}u(y) )\geq 0.\nonumber
			\end{align}
			
			When $u(x)\geq 0$, $u(y)\leq 0$, or $u(y)\geq 0$, $u(x)\leq 0$, \eqref{eq2:AppendixA} is immediate.\par 
			Otherwise, if $u(x)\geq u(y)$ such that $ |u(x)|\geq |u(y)|$, \eqref{eq2:AppendixA} still holds
			\begin{align*}
				&(u(x)-u(y))( |u(x)|^{\rho}u(x)-|u(y)|^{\rho}u(y) )\geq 0
				\\&\iff |u(x)|^{\rho}u(x)-|u(y)|^{\rho}u(y)\geq  0
				\\&\iff |u(x)|^{\rho}u(x)\geq |u(y)|^{\rho}u(y).
			\end{align*}
			
			In the case that $u(x)$ and $u(y)$ are negative values, satisfying $u(x)< u(y)$ and $ |u(x)|< |u(y)|$, then
			\[\frac{u(x)}{u(y)}> 1> \left(\frac{|u(y)|}{|u(x)|}\right)^{\rho}.\]
			Thus, \eqref{eq2:AppendixA} is satisfied since
			\begin{align*}
				&(u(x)-u(y))( |u(x)|^{\rho}u(x)-|u(y)|^{\rho}u(y) )> 0
				\\ &\iff |u(x)|^{\rho}u(x)-|u(y)|^{\rho}u(y) <0
				\\&\iff \frac{u(x)}{u(y)}>\left(\frac{|u(y)|}{|u(x)|}\right)^{\rho}.
			\end{align*}		
		
			\item If $|u(x)|\leq k$, $u(y)<-k$.\par 
			In this case, \eqref{eq:AppendixA} is achieved if
			\begin{align}\label{eq3:AppendixA}
				&(u(x)-u(y))(|T_k(u(x))|^{\rho}T_k(u(x))-|T_k(u(y))|^{\rho}T_k(u(y)))
				\\&=(u(x)-u(y))(|u(x)|^{\rho}u(x))-(-k^{\rho}k))>0.\nonumber
			\end{align}
			
			If $u(x)=0$, then $(u(x)-u(y))((|u(x)|^{\rho}u(x))-(-k^{\rho}k))> 0$.\par 
			On the other hand, since $u(x)\geq -k>u(y)$, 
			\[\frac{u(x)}{-k}\leq 1\leq \left(\frac{k}{|u(x)|}\right)^{\rho}.\]
			Then \eqref{eq3:AppendixA} is satisfied,
			\begin{align*}
				&(u(x)-u(y))((|u(x)|^{\rho}u(x))-(-|k|^{\rho}k))\geq 0
				\\&\iff|u(x)|^{\rho}u(x)+ k^{\rho}k\geq 0
				\\&\iff \frac{u(x)}{-k}\leq \left(\frac{k}{|u(x)|}\right)^{\rho}.
			\end{align*}
		\end{enumerate}
	\section*{Acknowledgements}
	The authors would like to thank the anonymous referees for their useful comments and suggestions, which have helped to improve the quality of this paper. This work is supported by Junta de Andaluc\'ia FQM--424 (Spain). The first author acknowledges additional support from University of Almeria's programme for research and knowledge transfer, "Plan Propio de Investigaci\'on y Transferencia de la Universidad de
	Almer\'ia".
		
		\section*{Data availability}
		No data was used for the research described in the article.

	\end{document}